\documentclass[12pt]{article} 
\usepackage{amsfonts}

\usepackage{amsmath}

\newtheorem{theorem}{Theorem}

\newtheorem{lemma}[theorem]{Lemma}

\newenvironment{proof}[1][Proof]{\textbf{#1.} }{\ \rule{0.5em}{0.5em}}

\begin{document}

\title{Regularity of the density for the stochastic heat equation }
\author{Carl Mueller$^1$  \\
Department of Mathematics\\
University of Rochester\\
Rochester, NY 15627  USA\\
email: cmlr@math.rochester.edu 
\and 
David Nualart$^2$ \\
Department of Mathematics\\
University of Kansas\\
Lawrence, Kansas, 66045 USA\\
email: nualart@math.ku.edu}
\maketitle

\footnotetext[1]{Supported by NSA and NSF grants.}
\footnotetext[2]{Supported by NSF grant DMS-0604207.


}

\begin{abstract}
We study the smoothness of the density of a semilinear heat equation with 
multiplicative spacetime white noise.  Using Malliavin calculus, we reduce the 
problem to a question of negative moments of solutions of a linear heat 
equation with multiplicative white noise.  Then we settle this question by 
proving that solutions to the linear equation have negative moments of all 
orders. 
\end{abstract}

\section{Introduction}

Consider the one-dimensional stochastic heat equation on $[0,1]$ with
Dirichlet boundary conditions, driven by a two-parameter
white noise, and with initial condition $u_{0}$:
\begin{equation}
\frac{\partial u}{\partial t}=\frac{\partial ^{2}u}{\partial x^{2}}%
+b(t,x,u(t,x))+\sigma (t,x,u(t,x))\frac{\partial ^{2}W}{\partial t\partial x}%
.  \label{e1}
\end{equation}%
Assume that the coefficients $b(t,x,u),\sigma(t,x,u)$ have linear growth in 
$t,x$ and are Lipschitz functions of $u$, uniformly in $(t,x)$. 

In \cite{PZ} Pardoux and Zhang proved that $u(t,x)$ has an absolutely
continuous distribution for all $(t,x)$ $\ $such that $t>0$ and $x\in (0,1)$%
, if $\sigma (0,y_0,u_{0}(y_0))\not=0$ for some $y_0\in (0,1)$. Bally and Pardoux have
studied the regularity of the law of the solution of Equation (\ref{e1})
with Neumann boundary conditions on $[0,1]$, assuming that the coefficients $%
b(u)$ and $\sigma(u) $ are infinitely differentiable functions, which are bounded
together with their derivatives. 

Let $u(t,x)$ be the solution of Equation (\ref{e1}) with Dirichlet boundary
conditions on $[0,1]$ and assume that the coefficients $b$ and $\sigma $ are
infinitely differentiable functions of the variable $u$ with bounded derivatives. The aim of
this paper is to show that if $\sigma (0,y_0,u_{0}(y_0))\not=0$ for some $y_0\in
(0,1)$, then $u(t,x)$ has a smooth density \ for all $(t,x)$ $\ $such that $%
t>0$ and $x\in (0,1)$. Notice that this is exactly the same nondegeneracy
condition \ imposed in \cite{PZ} to establish the absolute continuity.  In
order to show this result we make use of a general theorem on the existence
of negative moments for the solution of Equation (\ref{e1}) in the case $%
b(t,x,u)=B(t,x)u$ and $\sigma (t,x,u)=H(t,x)u$, where $B$ and $H$ are some
bounded and adapted random fields.

\section{Preliminaries}

First we define white noise $W$.  
Let  
\[
W=\{W(A),A \mbox{ a Borel subset of } \mathbb{R}^{2},|A|<\infty \}
\]
be a Gaussian family of random variables with zero mean and covariance%
\begin{equation*}
E\big[W(A)W(B)\big]=|A\cap B|,
\end{equation*}%
where $|A|$ denotes the Lebesgue measure of a Borel subset of $\mathbb{R}%
^{2} $, defined on a complete probability space $(\Omega ,\mathcal{F},P)$.
Then $W(t,x)=W([0,t]\times \lbrack 0,x])$ defines a two-parameter Wiener
process on $[0,\infty )^{2}$.

We are interested in the following one-dimensional heat equation on $%
[0,\infty )\times \lbrack 0,1]$%
\begin{equation}
\frac{\partial u}{\partial t}=\frac{\partial ^{2}u}{\partial x^{2}}%
+b(t,x,u(t,x))+\sigma (t,x,u(t,x))\frac{\partial ^{2}W}{\partial t\partial x}%
,  \label{a1}
\end{equation}%
with initial condition $u(0,x)=u_{0}(x)$, and Dirichlet boundary conditions $%
u(t,0)=u(t,1)=0$. We will assume that $u_{0}$ is a continuous function which
satisfies the boundary conditions $u_{0}(0)=u_{0}(1)=0$. This equation is
formal because the partial derivative 
$\frac{\partial ^{2}W}{\partial t\partial x}\,\ $\ 
does not exist, and (\ref{a1}) is usually replaced by the
evolution equation 
\begin{eqnarray}
u(t,x)
&=&\int_{0}^{1}G_{t}(x,y)u_{0}(y)dy+\int_{0}^{t}%
\int_{0}^{1}G_{t-s}(x,y)b(s,y,u(s,y))u(s,y)dyds  \notag \\
&&+\int_{0}^{t}\int_{0}^{1}G_{t-s}(x,y)\sigma (s,t,u(s,y))u(s,y)W(dy,ds),
\label{a2}
\end{eqnarray}%
where $G_{t}(x,y)$ is the fundamental solution of the heat equation 
on $[0,1]$ with Dirichlet boundary conditions.  Equation (\ref{a2}) is called 
the mild form of the equation.  

If the coefficients $b$ and $\sigma \,\ $are   have linear growth and are
Lipschitz functions of $u$, uniformly in $(t,x)$, there exists a unique
solution of Equation (\ref{a2}) (see Walsh \cite{Wa}).

The Malliavin calculus is an infinite dimensional calculus on a Gaussian
space, which is mainly applied to establish the regularity of the law of
nonlinear functionals of the underlying Gaussian process. We will briefly
describe the basic criteria for existence and smoothness of densities, and we
refer to Nualart \cite{Nu} for a more complete presentation of this subject.

Let $\mathcal{S}$ denote the class of smooth random variables of the the
form 
\begin{equation}
F=f(W(A_{1}),\dots ,W(A_{n})),  \label{a5}
\end{equation}%
where $f$ belongs to $C_{p}^{\infty }(\mathbb{R}^{n})$ ($f$ and all its
partial derivatives have polynomial growth order), and $A_{1},\dots ,A_{n}$
are Borel subsets of $\mathbb{R}_{+}^{2}$ with finite Lebesgue measure. The
derivative of $F$ is the two-parameter stochastic process defined by%
\begin{equation*}
D_{t,x}F=\sum_{i=1}^{n}\frac{\partial f}{\partial x_{i}}(W(A_{1}),\dots
,W(A_{n}))\mathbf{1}_{A_{i}}(t,x).
\end{equation*}%
In a similar way we define the iterated derivative $D^{(k)}F$. The derivative
operator $D$ (resp. its iteration $D^{(k)}$)$\ $ is a closed operator from $%
L^{p}(\Omega )$ into $L^{p}(\Omega ;L^{2}(\mathbb{R}^{2}))$ (resp. $%
L^{p}(\Omega ;L^{2}(\mathbb{R}^{2k}))$) for any $p>1$. For any $p>1$ and for
any positive integer $k$ we denote by $\mathbb{D}^{p,k}$ the completion of $%
\mathcal{S}$ with respect to the norm%
\begin{equation*}
\left\| F\right\| _{k,p}=\left\{ E(|F|^{p})+\sum_{j=1}^{k}E\left[ \left(
\int_{\mathbb{R}^{2j}}\left( D_{z_{1}}\cdots D_{z_{j}}F\right)
^{2}dz_{1}\cdots dz_{j}\right) ^{\frac{p}{2}}\right] \right\} ^{\frac{1}{p}}.
\end{equation*}%
Set $\mathbb{D}^{\infty }=\cap _{k,p}\mathbb{D}^{k,p}$.

Suppose that $F=(F^{1},\ldots ,F^{d})$ is a $d$-dimensional random vector
whose components are in $\mathbb{D}^{1,2}$. Then, we define the Malliavin
matrix of $F$ as the random symmetric nonnegative definite matrix%
\begin{equation*}
\sigma _{F}=\left( \left\langle DF^{i},DF^{j}\right\rangle _{L^{2}(\mathbb{R}%
^{2})}\right) _{1\leq i,j\leq d}.
\end{equation*}%
The basic criteria for the existence and regularity of the density are the
following:

\begin{theorem}
\label{t1} Suppose that $F=(F^{1},\ldots ,F^{d})$ is a $d$-dimensional
random vector whose components are in $\mathbb{D}^{1,2}$. Then,

\begin{enumerate}
\item If $\det \sigma _{F}>0$ almost surely, the law of $F$ is absolutely
continuous.

\item If $F^{i}\in \mathbb{D}^{\infty }$ for each $i=1,\ldots ,d$ and $E%
\left[ (\det \sigma _{F})^{-p}\right] <\infty $ for all $p\geq 1$, then the $%
F$ has an infinitely differentiable density.
\end{enumerate}
\end{theorem}

\section{Negative moments}

\begin{theorem}  
\label{t2}
Let $u(t,x)$ be the solution to the stochastic heat equation%
\begin{eqnarray}
\frac{\partial u}{\partial t}&=&\frac{\partial ^{2}u}{\partial x^{2}}+Bu+Hu%
\frac{\partial ^{2}W}{\partial t\partial x},  \label{e3}  \\
u(0,x) &=& u_0(x)  \notag
\end{eqnarray}%
on $x\in[0,1]$ with Dirichlet boundary conditions.  Assume that 
$B=B(t,x)$ and $H=H(t,x)$ are bounded and adapted processes. Suppose 
that $u_{0}(x)\  $ is a nonnegative continuous function not identically zero. 
Then, 
\begin{equation*}
E[u(t,x)^{-p}]<\infty
\end{equation*}
for all $p\geq 2$, $t>0$ and $0<x<1$.
\end{theorem}

For the proof of this theorem we will make use of the following large 
deviations lemma, which follows from Proposition A.2, page 530, of Sowers
\cite{So}. 

\begin{lemma} \label{lem1}
Let $w(t,x)$ be an adapted stochastic process, bounded in absolute value by a 
constant $M$.   Let $\epsilon>0$.  Then, there exist constants $C_0$,  $C_1>0$ 
such that for  all $\lambda>0$  and all $T>0$ 
\[
P\left( \sup_{0\le t\le T} \sup _{0\le x\le 1} \left| \int_0^t \int_0^1 G_{t-s} (x,y) w(s,y) W(ds,dy) \right| >\lambda \right) \le C_0 \exp \left( -\frac {C_1 \lambda^2} {T^{\frac 12-\epsilon} }\right).
\]
\end{lemma}

We also need a comparison theorem such as Corollary 2.4 of \cite{Sh};  
see also Theorem 3.1 of Mueller \cite{Mu} or Theorem 2.1 of Donati-Martin and 
Pardoux \cite{DP}.  
Shiga's result is for $x\in\mathbb{R}$, but it can easily be extended to the 
following  lemma, which deals with $x\in[0,1]$ and Dirichlet boundary 
conditions. 

\begin{lemma} 
\label{lem-compare}
Let $u_i(t,x):  i=1,2$ be two solutions of 
\begin{eqnarray}
\frac{\partial u_i}{\partial t}&=&\frac{\partial ^{2}u_i}{\partial x^{2}}+B_iu_i+Hu_i%
\frac{\partial ^{2}W}{\partial t\partial x},  \label{e3a}  \\
u_i(0,x) &=& u_0^{(i)}(x)  \notag
\end{eqnarray}%
where $B_i(t,x), H(t,x), u^{(i)}_0(x)$ satisfy the same conditions as in 
Theorem \ref{t2}.  Also assume that with probability one for all 
$t\geq0,x\in[0,1]$
\begin{eqnarray*}
B_1(t,x) &\leq& B_2(t,x)  \\
u_0^{(1)}(x) &\leq& u_0^{(2)}(x).
\end{eqnarray*}
Then with probability 1, for all $t\geq0, x\in[0,1]$.  
\[
u_1(t,x)\leq u_2(t,x).
\]
\end{lemma}

\begin{proof}[Proof of Theorem \ref{t2}]
We shall repeatedly use the comparison lemma, Lemma \ref{lem-compare}, along 
with the  following argument.    
Observe that if $0<w(t,x)\leq u(t,x)$ with probability one, and $p>0$, then
\[
E\left[u(t,x)^{-p}\right] \leq E\left[w(t,x)^{-p}\right].
\]
Thus, to bound $E[u(t,x)^{-p}]$, it suffices to find a nonnegative function 
$w(t,x)\leq u(t,x)$ and to prove a bound for $E[w(t,x)^{-p}]$.  Such a 
function $w(t,x)$ might be found using the comparison lemma, Lemma 
\ref{lem-compare}.

Suppose that $|B(t,x)|\leq K$ almost surely for some constant $K>0$. By the 
comparison lemma, Lemma \ref{lem-compare}, it suffices to 
consider the solution to the equation%
\begin{eqnarray}
\frac{\partial w}{\partial t}&=&\frac{\partial ^{2}w}{\partial x^{2}}-Kw+Hw%
\frac{\partial ^{2}W}{\partial t\partial x}  \label{e4}  \\
w(0,x) &=& u_0(x)  \notag
\end{eqnarray}%
on $x\in[0,1]$ with Dirichlet boundary conditions.  
Indeed, the comparison lemma implies that a solution $w(t,x)$ of 
(\ref{e4}) will be less than or equal to a solution $u(t,x)$ of (\ref{e3}). 
Then we can use the argument outlined in the previous
paragraph to conclude that the boundedness of 
$E\left[w(t,x)^{-p}\right]$ implies the boundedness of 
$E\left[u(t,x)^{-p}\right]$.  

Set $u(t,x)=e^{-Kt}w(t,x)$, where $u(t,x)$ is not the same as earlier in the 
paper. Simple calculus shows that $u(t,x)$ satisfies%
\begin{eqnarray}
\frac{\partial u}{\partial t}&=&\frac{\partial ^{2}u}{\partial x^{2}}+Hu\frac{%
\partial ^{2}W}{\partial t\partial x}.  \label{e4b}  \\
u(0,x) &=& u_0(x) \notag
\end{eqnarray}%
and we have
\[
E\left[w(t,x)^{-p}\right] = e^{Ktp}E\left[u(t,x)^{-p}\right].
\]
So, we can assume that $K=0$, that is that $u(t,x)$ satisfies (\ref{e4b}).
The mild formulation of Equation (\ref{e4b})  is%
\begin{equation*}
u(t,x)=\int_{0}^{1}G_{t}(x,y)u_{0}(y)dy+\int_{0}^{t}%
\int_{0}^{1}G_{t-s}(x,y)H(s,y)u(s,y)W(ds,dy).
\end{equation*}%
 
Suppose that $u_{0}(x)\geq \delta >0$ for all $x\in \lbrack a,b]\subset (0,1)
$. Since (\ref{e4b}) is linear, we may divide this equation by $\delta $,
and assume $\delta =1$, and also $u_{0}(x)=\mathbf{1}_{[a,b]}(x)$.   Fix $T>0$, and
consider a larger interval $[a,b]\subset [c,d]$ of the form $d=b+\gamma T$ and 
$c=a-\gamma T$, where $\gamma>0$.  We are going to show that $E((u(T,x)^{-p})<\infty$ for 
$x\in [c,d]$ and for any $p\ge 1$.
Define 
\[
c= \inf_{0\le t+s \le T,} \inf_{a-\gamma (t+s)\le x\le  b+\gamma (t+s)} \int_{a-\gamma s}^{b+\gamma s} G_t(x,y)dy
\]
and note that $0<c<1$ for each $\gamma>0$ and $(a,b)\in(0,1)$.  
Next we inductively define a sequence $\left\{ \tau _{n},n\geq 0\right\} $
of stopping times and a sequence of processes $v_{n}(t,x)$ as follows.
 Let $v_{0}(t,x)$ be the solution of (\ref%
{e4b}) with initial condition $u_{0}=\mathbf{1}_{[a,b]}$ and let 
\[
\tau_0= \inf\left\{t>0:  \inf_{a-\gamma t \le x\le b+\gamma t} 
v_0(t,x)=\frac{c}{2}  \, \mbox{ or }\,  \sup_{0\le x\le 1} v_0(t,x)= \frac{2}{c}\right\}.
\]
Next, assume that we have defined $\tau_{n-1}$ and  $v_{n-1}(t,x)$ for $\tau_{n-2}\le t\le\tau_{n-1}$. Then, $\{v_n(t,x), \tau_{n-1}\le t\}$ is defined by (\ref{e4b}) with initial condition $v_n(\tau_{n-1},x)=(\frac c2)^{n} \mathbf{1}_{[a-\gamma \tau_{n-1}, b+\gamma \tau_{n-1}]}(x)$.  Also, let
\begin{eqnarray*}
\lefteqn{ 
\tau_n= \inf\Bigg\{t>\tau_{n-1}:  \inf_{a-\gamma t \le x\le b+\gamma t} 
v_n(t,x)=\left(\frac c2\right)^{n+1}  }\\
&& \hspace{1.5in} \mbox{ or } \sup_{0\le x \le 1} 
v_n(t,x)=\left(\frac 2c\right)^{-n+1}\Bigg\}.
\end{eqnarray*}
It is not hard to see that $\tau_n <\infty$ almost surely.
Notice that
\[
\inf_{a-\gamma  \tau_n \le x \le b+\gamma \tau_n} v_n(\tau_n ,x) \ge \left(\frac c2\right)^{n+1}.
\]
By the comparison lemma, we have that
\begin{equation}
u(t,x)\ge v_n(t,x)  \label{e5}
\end{equation}
for all $(t,x)$ and all $n\ge 0$.  
For all $p\ge 1$ we have
\begin{eqnarray}
E\left[u(T,x)^{-p}\right]
&\leq& P\Big(u(T,x)\geq 1\Big)  \notag  \\
&&      + \sum_{n=0}^{\infty}\left( \frac 2c\right)^{np}P\left(u(T,x)\in\Big[ \left(\frac c2\right)^{n+1},\left(\frac c2\right)^n\Big)\right)  \notag\\
&\leq& 1 +  \sum_{n=0}^{\infty}\left( \frac 2c\right)^{np}  P\left(u(T,x) < \left(\frac c2\right)^n \right).
\label{e44}
\end{eqnarray}
 Taking into account  (\ref{e5}), the event $\{ u(T,x) <(\frac c2)^n\}$ is included in $\mathcal{A}_n=\{\tau_n <T\}$. Set $\sigma_n=\tau_n- \tau_{n-1}$, for all $n\ge 0$, with the convention $\tau_{-1}=0$.

We have
\begin{eqnarray*}
  P\left(\sigma_i <\frac 2n \bigg| \mathcal{F}_{\tau_{i-1}}\right) &\le& 
  P\left( \sup_{\tau_{i-1} <t <\tau_{i-1} +\frac 2n,} \sup_{0\le x\le 1}   v_i(t,x) >\left(\frac 2c\right)^{-i+1}\right) \\
&&  + P\left( \inf_{\tau_{i-1} <t <\tau_{i-1} +\frac 2n,}  \inf_{a-\gamma t \le x\le b+\gamma t}  v_i(t,x) >\left(\frac c2\right)^{i+1}   \right) \\
\end{eqnarray*}
Notice that, for $\tau_{i-1} <t< \tau_i$ we have
\begin{eqnarray*}
\lefteqn{
\left(\frac 2c\right)^iv_i(t,x) =  \int_{a-\gamma \tau_{i-1}}^{b+\gamma \tau_{i-1}} G_{t-\tau_{i-1}}(x,y) dy }\\
&&+\int_{\tau_{i-1}} ^t \int_0^1 G_{t-s}(x,y) H(s,y) 
 \left(\left[\left(\frac 2c\right)^i v_i(s,y)\right]\wedge 
  \frac2c \right)W(ds,dy).
\end{eqnarray*}
As a consequence, by Lemma \ref{lem1}
\begin{eqnarray}
 P\left( \sigma_i <\frac 2n \bigg| \mathcal{F}_{\tau_{i-1}}\right) &  \le & P\left( \sup_{\tau_{i-1} \le t\le \tau_{i-1}+\frac 2n,}   \notag
\sup_{0\le x\le 1} \left| N_i(t,x)\right| > 1 \right)  \\
&\le&  C_0 \exp \left( - C_1  n^{\frac 12 -\epsilon} \right).  \label{e11}
\end{eqnarray}

Next we set up some notation.  Let   $\mathcal{B}_n$ be the event that at least half of the variables 
$\sigma_i: i=0,\ldots,n$ satisfy
\[
\tau_i < \frac{2T}{n}
\]
Note that
\[
 \mathcal{A}_n \subset \mathcal{B}_n
\]
since if more than half of the $\sigma_i: i=1,\ldots,n$ are larger than or equal 
to $2T/n$ then $\tau_n>T$.   

For convenience we assume that $n=2k$ is even, and leave the odd case to the 
reader.  Let $\Xi_n$ be all the subsets of $\{1,\ldots,n\}$ of cardinality 
$k=n/2$.  Using Stirling's formula, the reader can verify that as $n\to\infty$
\begin{equation}
{n \choose n/2} = O(2^n)    \label{e2}
\end{equation}
Then,
\begin{eqnarray*}
P(\mathcal{B}_n)
&\leq& P\left(\bigcup_{\{i_1,\ldots,i_k\}\in\Xi_n}\bigcap_{j=1}^{k}
             \left\{\sigma_{i_j}<\frac{2T}{n}\right\}\right)    \\
&\leq&  \sum_{\{i_1,\ldots,i_k\}\in\Xi_n} P\left(\bigcap_{j=1}^{k} \sigma_{i_j}<\frac{2T}{n}\right) 
\end{eqnarray*}
Using the estimate (\ref{e11})  and (\ref{e2}) yields
 \begin{eqnarray*}
P(\mathcal{B}_n)&\leq&  C_0 2^{n} \exp\left(-C_1n^{1/2-\varepsilon}\right)^n  \\
&\leq&  C_0 \exp\left(-C_1 n^{3/2-\varepsilon}+C_2n\right)   \\
&\leq&  C_0 \exp\left(-C_1 n^{3/2-\varepsilon}\right) 
\end{eqnarray*}
where the constants $C_0,C_1$ may have changed from line to line.  
Hence,
\begin{equation}
P\left(  u(T,x) <\left(\frac c2\right)^n\right) \le C_0 \exp\left(-C_1 n^{3/2-\varepsilon} \right)
\label{e33}
\end{equation}
Finally, substituting  (\ref{e3}) into  (\ref{e4}) yields $E\left[u(T,x)^{-p}\right]<\infty$.
\end{proof}

\section{Smoothness of the density}

Let $u(t,x)$ be the solution to Equation (\ref{a1}). Assume that the
coefficients $b$ and $\sigma $ are continuously differentiable with bounded
derivatives. Then $u(t,x)$ belongs to the Soboev space $\mathbb{D}^{1,p}$
for all $p>1$, and the derivative $D_{\theta ,\xi }u(t,x)$ satisfies the
following evolution equation%
\begin{eqnarray}
D_{\theta ,\xi }u(t,x) &=&\int_{\theta
}^{t}\int_{0}^{1}G_{t-s}(x,y)\frac{\partial b}{\partial u}(s,y,u(s,y))D_{\theta ,\xi
}u(s,y)dyds  \notag \\
&&+\int_{\theta }^{t}\int_{0}^{1}G_{t-s}(x,y)\frac{\partial\sigma}{\partial u}
(s,y,u(s,y))D_{\theta ,\xi }u(s,y)W(dy,ds)  \notag \\
&&+\sigma (u(\theta ,\xi ))G_{t-\theta }(x,\xi ),  \label{e22}
\end{eqnarray}%
if $\theta <t$ and $D_{\theta ,\xi }u(t,x)=0$ if $\theta >t$. That is, $%
D_{\theta ,\xi }u(t,x)$ is the solution of the stochastic partial
differential equation%
\begin{equation*}
\frac{\partial D_{\theta ,\xi }u}{\partial t}=\ \frac{\partial ^{2}D_{\theta
,\xi }u}{\partial x^{2}}+ \frac{\partial b}{\partial u}(t,x,u(t,x))  D_{\theta ,\xi }u
+ \frac{\partial\sigma}{\partial u}
(t,x,u(t,x))D_{\theta ,\xi }u\frac{\partial ^{2}W}{\partial t\partial x}
\end{equation*}%
on $[\theta ,\infty )\times \lbrack 0,1]$, with Dirichlet boundary
conditions and initial condition $\sigma (u(\theta ,\xi ))\delta _{0}(x-\xi )
$.

\begin{theorem}
Let $u(t,x)$ be the solution of Equation (\ref{a1}) with initial condition $%
u(0,x)=u_{0}(x)$, and Dirichlet boundary conditions $u(t,0)=u(t,1)=0$. We
will assume that $u_{0}$ is an  $\alpha$-H\"older  continuous function for some $\alpha>0$,   which satisfies the boundary conditions $%
u_{0}(0)=u_{0}(1)=0$. Assume that the coefficients $b$ and $\sigma $ are
infinitely differentiable functions with bounded derivatives. Then, if $%
\sigma (0,y_{0},u_{0}(y_{0}))\not=0$ for some $y_{0}\in (0,1)$, $u(t,x)$ has
a smooth density \ for all $(t,x)$ $\ $such that $t>0$ and $x\in (0,1)$.
\end{theorem}

\begin{proof}
From the results proved by Bally and Pardoux in \cite{BP} we know that $%
u(t,x)$ belongs to the space $\mathbb{D}^{\infty }$ for all $(t,x)$. Set $\ $%
\begin{equation*}
C_{t,x}=\int_{0}^{t}\int_{0}^{1}\left( D_{\theta ,\xi }u(t,x)\right)
^{2}d\xi d\theta .
\end{equation*}%
Then, by Theorem \ref{t1} it suffices to show that 
$E(C_{t,x}^{-p})<\infty$ 
for all $p\geq 2$. 

Suppose that $\sigma (0,y_{0},u_{0}(y_{0}))>0$. By continuity we have that $%
\sigma (0,y,u(0,y))\geq \delta >0$ for all $y\in \lbrack a,b]\subset (0,1)$.
Then
\begin{eqnarray*}
C_{t,x} \ge  \int_{0}^{t}\int_{a}^{b}\left( D_{\theta ,\xi }u(t,x)\right)
^{2}d\xi d\theta
\ge
\int_{0}^{t} \left(\int_{a}^{b}  D_{\theta ,\xi }u(t,x)  d\xi\right)
^{2} d\theta.
\end{eqnarray*}
Set  $Y^{\theta}_{t,x}=\int_{a}^{b}D_{\theta,\xi }u(t,x)d\xi $.
Fix $r<1$ and $\varepsilon >0$ such that $\varepsilon
^{r}<t$.  
From%
\begin{equation*}
\ \varepsilon ^{r}  \left( Y^0_{t,x}\right)^2  \leq \
  \int_{0}^{\varepsilon ^{r}}   \left| \left( Y^0_{t,x}\right)^2-\left( Y^\theta_{t,x}\right)^2
\right| d\theta   +C_{t,x}
\end{equation*}%
we get%
\begin{eqnarray*}
P(C_{t,x} <\varepsilon )&\leq& P\left(   \int_{0}^{\varepsilon
^{r}}\left| \left( Y^0_{t,x}\right)^2-\left( Y^\theta_{t,x}\right)^2
\right| d\theta   >\varepsilon %
\right)  \\
&&+P\left(  Y^0_{t,x}  < 
\sqrt{2} \varepsilon ^{\frac{1-r}2}  \right)  \\
&=&P(A_{1})+P(A_{2}).
\end{eqnarray*}%
Integrating \ equation (%
\ref{e22}) in the variable $\xi $ yields the following equation for the process
$\{Y^\theta_{t,x}, t\ge \theta, x\in [0,1]\}$
\begin{eqnarray}
Y^\theta_{t,x} &=&\int_{\theta}^{t}\int_{0}^{1}G_{t-s}(x,y)\frac{\partial b}{\partial u}
(s,y,u(s,y))Y^\theta_{s,y}dyds  \notag \\
&&+\int_{\theta}^{t}\int_{0}^{1}G_{t-s}(x,y)\frac{\partial\sigma}{\partial u}
(s,y,u(s,y))Y^\theta_{s,y}W(dy,ds)  \notag \\
&&+\int_{a}^{b}\sigma ( u(\theta,\xi ))G_{t-\theta}(x,\xi )d\xi .  \label{eq3}
\end{eqnarray}%
In particular, for $\theta=0$, the initial condition is $Y^0_{0,\xi}=\sigma(0,\xi,u(0,\xi)) \mathbf{1}_{[a,b]}(\xi)$, and by Theorem \ref{t2} the random variable $Y^0_{t,x}$ has negative moments of all orders. Hence, for all $p\ge 1$,
\[
P(A_2) \le \varepsilon^p
\]
if $\varepsilon \le \varepsilon_0$. In order to handle the probability $P(A_1)$ we write  
\begin{equation*}
P(A_{1})\leq  \varepsilon ^{(r-1)q}\sup_{  0\leq \theta \leq
\varepsilon ^{r} } \left(E\left[   \left| Y^\theta_{t,x} -Y^0_{t,x} \right|^{2q}
\right] E\left[   \left| Y^\theta_{t,x} + Y^0_{t,x} \right|^{2q}
\right] \right)^{1/2}.
\end{equation*}%
We claim that   
\begin{equation} \label{c1}
\sup_{0\le \theta \le t} E\left[ \left| Y^{\theta}_{x,t} \right|^{2q} \right] <\infty,
\end{equation}
and 
\begin{equation}  \label{c2}
\sup_{0\le \theta \le \varepsilon^r } E\left[ \left| Y^\theta_{t,x} -Y^0_{t,x}  \right|^{2q} \right] <\varepsilon^{2sq},
\end{equation}
for some $s>0$.  Property (\ref{c1}) follows easily from Equation  (\ref{eq3}). 
On the other hand, the difference  $Y^\theta_{t,x} -Y^0_{t,x} $ satisfies
\begin{eqnarray}
Y^\theta_{t,x} -Y^0_{t,x} & =& \int_{\theta}^{t}\int_{0}^{1}G_{t-s}(x,y)
\frac{\partial b}{\partial u}
(s,y,u(s,y))(Y^\theta_{s,y} -Y^0_{x,t}) dyds  \notag \\
&&+\int_{\theta}^{t}\int_{0}^{1}G_{t-s}(x,y) \frac{\partial\sigma}{\partial u}
(s,y,u(s,y))(Y^\theta_{s,y}- Y^0_{x,t}) W(dy,ds)  \notag \\
&& +\int_0^{\theta}\int_{0}^{1}G_{t-s}(x,y) \frac{\partial b}{\partial u}
(s,y,u(s,y))Y^0_{s,y}dyds  \notag \\
&& +\int_0^{\theta}\int_{0}^{1}G_{t-s}(x,y) \frac{\partial\sigma}{\partial u}
(s,y,u(s,y))Y^0_{s,y}W(dy,ds)  \notag \\
&& +\int_{a}^{b} (\sigma ( u(\theta,\xi )) G_{t-\theta}(x,\xi ) -\sigma(u_0(\xi)) G_{t }(x,\xi ))d\xi \notag\\
&=& \sum_{i=1} ^5 \Psi_i(\theta). \notag
\end{eqnarray}%
Applying Gronwall's lemma and standard estimates, to show (\ref{c2}) it suffices to prove that
\begin{equation}  \label{c3}
\sup_{0\le \theta \le \varepsilon^r } E\left( \left|  \Psi_i(\theta)  \right|^{2q} \right) <\varepsilon^{2sq},
\end{equation}
for $i=3,4,5$ and for some $s>0$.  The estimate (\ref{c3}) is clear for $i=3,4$ and for $i=5$ we use the properties of the heat kernel and the H\"older continuity of the initial condition $u_0$. 
Finally, it suffices to choose $r>1-s$ and we get the desired estimate for $P(A_1)$. The proof is now complete.
\end{proof}

\end{document}